\documentclass[]{interact}

\usepackage{epstopdf}
\usepackage[caption=false]{subfig}

\usepackage[numbers,sort&compress]{natbib}
\usepackage{algorithm,algorithmic}
\bibpunct[, ]{[}{]}{,}{n}{,}{,}
\makeatletter
\def\NAT@def@citea{\def\@citea{\NAT@separator}}
\makeatother

\theoremstyle{plain}
\newtheorem{theorem}{Theorem}[section]
\newtheorem{lemma}[theorem]{Lemma}

\theoremstyle{definition}

\theoremstyle{remark}

\newcommand{\MR}{\mathbb{R}}

\begin{document}


\title{A New Uzawa-exact Type Algorithm for Nonsymmetric Saddle Point Problems}

\author{
\name{Zhitao Xu\textsuperscript{a}\thanks{CONTACT Zhitao Xu. Email: xuzhitao@pku.edu.cn}, Ting Jiang\textsuperscript{a}\thanks{CONTACT Ting Jiang. Email: jiangting@ccbtrust.com.cn} and Li Gao\textsuperscript{a}\thanks{CONTACT Li Gao. Email: gao@math.pku.edu.cn}}
\affil{\textsuperscript{a}School of Mathematical Sciences, Peking University, Beijing, 100871, China}
}

\maketitle

\begin{abstract}
Saddle point problems have been attracting people's attention in recent years. To solve large and sparse saddle point problems, Uzawa type algorithms were proposed. The main contribution of this paper is to present a new Uzawa-exact type algorithm from the aspect of optimization method to solve
nonsymmetric saddle point problems, which often arise from linear variational inequalities and finite element discretization of Navier-Stokes equations. In the paper, convergence of the new algorithm is analysed and numerical experiments are presented.
\end{abstract}


\section{Introduction}
\label{sec:intro}
The linear saddle point problem is expressed as
\begin{equation}\label{eq:prob1}
\left(
    \begin{array}{cc}
      A & B_{1}^{T} \\
      B_{2} & -C \\
    \end{array}
  \right)
\left(
         \begin{array}{c}
           x \\
           y \\
         \end{array}
       \right)=\left(
                 \begin{array}{c}
                   f \\
                   h \\
                 \end{array}
               \right),
\end{equation}
where~$A\in \MR^{n\times n}$~,~$B_{1},B_{2} \in \MR^{m\times n}$~,~$C\in \MR^{m\times m}$~, and~$n\geq m$~. The coefficient matrix in (\ref{eq:prob1}) is called the Karush-Kuhn-Tucker (KKT) matrix.

The linear saddle point problems arise from, for example, variational inequalities, quadratic programming problems, finite element discretization of Navier-Stokes equations and Maxwell equations in electromagnetics \cite{taji1993globally,raviart1979finite,karakashian1982galerkin}.

Till now, many methods have been suggested to solve the saddle point problems. They can be subdivided into two broad categories \cite[see][p.29]{benzi2005numerical}: \emph{segregated methods}, which compute
$x$ and $y$ separately, and \emph{coupled methods}, which compute $x$ and $y$ simultaneously. The
main representatives of the segregated methods are the Schur complement reduction methods, the null space methods and the Uzawa type methods. The coupled methods mainly include
the Krylov subspace methods. The Uzawa type methods are
efficient to solve large and sparse problems, which were first suggested in 1958 \cite{arrow1958studies}, and much attention
has been paid since then.

For problem (\ref{eq:prob1}), the phrase {\it symmetric saddle point problem} is used when $A$ is a symmetric positive definite matrix, $B_1=B_2=B$, where $B$ is a full row rank matrix, and $C$ is a
symmetric positive semidefinite matrix, and the phrase {\it nonsymmetric saddle point problem} is used when $A$ is a non-symmetric positive definite matrix, $B_1=B_2=B$, where $B$ is a full row rank matrix, and $C$ is a symmetric positive semidefinite matrix.

The nonsymmetric saddle point problems that this paper mainly concern, often arise in linear variational inequality and certain discretization of Navier-Stokes equations, both of which we will introduce in numerical experiments.

Bramble \cite{bramble2000uzawa} first proposed two Uzawa type algorithms for solving the nonsymmetric saddle point problems with ~$C=0$~, and established their
convergence results. Later on, Cao \cite{cao2004fast} presented a nonlinear Uzawa type algorithm for solving the nonsymmetric saddle point problems with ~$C\neq0$~ and analysed the convergence of the algorithm.

The algorithms suggested by Bramble and Cao are considered from the aspect of numerical linear algebra. In their algorithms, some parameters have to be given by users.
In this paper, we are going to propose a new Uzawa-exact type algorithm from the aspect of optimization and analyse its convergence. The rest of this paper
is organized as follows. In Section \ref{sec:alg}, we describe a new algorithm for the nonsymmetric saddle point problems. In Section \ref{sec:convergence},
convergence results of the new algorithm are presented. And finally, in Section \ref{sec:exp}, numerical results are
reported by solving a set of test problems.

\section{A New Algorithm for the Nonsymmetric Saddle Point Problems}
\label{sec:alg}
In order to describe our new algorithm clearly, we start our brief description from the typical Uzawa algorithm for the symmetric saddle point problems, and then to present our new algorithm for the nonsymmetric saddle point problems.

The saddle point problem we considered here is described as
\begin{equation}\label{eq:symmetric}
\left[
    \begin{array}{cc}
      A & B^{T} \\
      B & -C \\
    \end{array}
\right]
\left[
         \begin{array}{c}
           x \\
           y \\
         \end{array}
\right]
=\left[
       \begin{array}{c}
            f \\
            h \\
       \end{array}
 \right],
\end{equation}
where
$A\in \MR^{n\times n}$ is positive definite,
$B\in \MR^{m\times n}$ is a full row rank matrix and $C\in \MR^{m\times m}$ is symmetric positive semidefinite.
The equation (\ref{eq:symmetric}) can be identically written as
\begin{subequations}\label{eq:symmetric1}
  \begin{align}
        Ax + B^{T}y=f,  \label{eq:symmetric1a}\\
        Bx - Cy = h.     \label{eq:symmetric1b}
  \end{align}
\end{subequations}
Getting $x$ from (\ref{eq:symmetric1a}) and substituting it into (\ref{eq:symmetric1b}), we have
\begin{equation}\label{eq:y1}
(BA^{-1}B^{T}+C)y=BA^{-1}f - h.
\end{equation}
Define
\begin{equation}\label{eq:Sb}
S=BA^{-1}B^{T}+C,\quad b=BA^{-1}f-h,
\end{equation}
then (\ref{eq:y1}) could be written as
\begin{equation}\label{eq:Sy=b}
Sy=b.
\end{equation}

When $A$ is a symmetric positive definite matrix, to solve the equation (\ref{eq:Sy=b}) is identical to solve the quadratic
optimization problem
\begin{equation}\label{eq:quadratic}
\min Q(y)=\frac{1}{2}y^{T}Sy-b^{T}y.
\end{equation}
The gradient of ~$Q(y)$~ is
\begin{equation*}
\nabla Q(y)= Sy-b=  -Bx + Cy + h.
\end{equation*}
By using the steepest descent method to solve the problem (\ref{eq:quadratic}), one avoids solving the equation (\ref{eq:Sy=b}) directly. The
so-called classical Uzawa algorithm is established in this way, and given in Algorithm \ref{alg:Uzawa}.
\begin{algorithm}
\caption{Classical Uzawa Algorithm}
\label{alg:Uzawa}
\begin{algorithmic}
\STATE{Given $x_{0}\in \MR^{n}$, $y_{0}\in \MR^{m}$ and $k:=0$}
\WHILE{a given stopping criterion is not satisfied}
\STATE{Solve $Ax_{k}=f-B^{T}y_{k}$}
\STATE{Set $y_{k+1}=y_{k}+\alpha (Bx_{k}-Cy_{k}-h)$}, where $\alpha>0$ is a stepsize
\STATE{Set $k:=k+1$}
\ENDWHILE
\end{algorithmic}
\end{algorithm}

For the nonsymmetric saddle point problems, Bramble \cite{bramble2000uzawa} first got the result that they are solvable if $A$ is invertible and the Ladyzhenskaya-Babu\v{s}ka-Brezzi condition \cite{brezzi2012mixed} holds, i.e. for some positive $c$, there is
\begin{equation}\label{eq:LBBcondition}
(BA^{-1}_sB^Tv,v)= \sup_{u\in \MR^n}\frac{(v,Bu)^2}{(A_su,u)}\geq c\|v\|^2, \quad \forall v\in \MR^m,
\end{equation}
where $A_s=\frac{1}{2}(A+A^T)$ is the symmetric part of $A$, which is positive definite.
We assume that the nonsymmetric saddle point problems considered here are solvable, and then propose a new algorithm for solving them based on Algorithm \ref{alg:Uzawa}.

The most important difference between the nonsymmetric and symmetric saddle point problems is that, in symmetric problems, $S$ is a symmetric positive definite matrix, so solving the linear system (\ref{eq:Sy=b})
is identical to solving the quadratic minimization problem (\ref{eq:quadratic}), where $S$ and $b$ are defined in (\ref{eq:Sb}). However, in the  nonsymmetric problems, $S$ is a nonsingular nonsymmetric matrix, the identification of the two problems does not exist any more. Therefore we have to find other ways to solve (\ref{eq:Sy=b}).

Our consideration is like this: instead of solving the system $Sy=b$, we solve the least squares problem
\begin{equation}\label{eq:quadratic2}
\min Q(y)=\frac{1}{2}(Sy-b)^{T}(Sy-b),
\end{equation}
where $S$ and $b$ are defined in (\ref{eq:Sb}). Now let us consider how to solve the problem (\ref{eq:quadratic2}). Notice that
the gradient of the objective function in (\ref{eq:quadratic2}) is ~$\nabla Q(y)=S^{T}(Sy-b)$~, the steepest descent method can not be used to solve (\ref{eq:quadratic2}) since we have to calculate $A^{-1}$ twice in every iteration, which
is high-cost. But considering the LBB condition (\ref{eq:LBBcondition}) and the positive semidefinite matrix $C$, we can find that $\forall \nu \neq 0$, there is
\begin{align*}
\nu^{T}S\nu&=\nu^{T}BA^{-1}B^{T}\nu+\nu^{T}C\nu \\
                                          &\geq(B^{T}\nu)^{T}A^{-1}B^{T}\nu \\
                                          &>0,
\end{align*}
that is the matrix $S$ is positive definite. This means if we choose $d=-(Sy-b)$, there is
\[
d^T \nabla Q(y)=-(Sy-b)^{T}S^{T}(Sy-b)<0,
\]
so $d$ is a descent direction at the point $y$. This direction
\begin{equation}\label{eq:direction}d=-(Sy-b)=Bx - Cy - h\end{equation}
could be chosen instead of $-\nabla Q(y)$ to avoid the computation of $A^{-1}$.

In the line search type algorithms of minimization problems, besides the descent direction, a suitable stepsize along the direction has to be considered.
In the classical Uzawa algorithms, how to choose the stepsize $\alpha$ is uncertain. Frequently, it was decided according to the background
of the application problems, or estimated by many numerical experiments. An advantage of our algorithm to solve the problem (\ref{eq:quadratic2}) is
that a proper stepsize could be acquired as an explicit formula along the direction.

For our problem, it is supposed that at the point $(x_{k},y_{k})$, the direction $d_k$ is decided by (\ref{eq:direction}), then we could use the exact line search method to get
the step size $\alpha_k$, which is a solution of the problem
\begin{equation}\label{eq:ELS}
\min\limits_{\alpha} Q(y_{k}+\alpha d_{k}),
\end{equation}
where
$$Q(y_{k}+\alpha d_{k})= \frac{1}{2}[ S(y_{k}+\alpha d_{k})-b]^{T}[ S(y_{k}+\alpha d_{k})-b].$$ By (\ref{eq:direction}),
we can get
\begin{align*}
  Q(y_{k}+\alpha d_{k}) &= \frac{1}{2}(\alpha Sd_k-d_k)^{T}(\alpha Sd_k-d_k)\\
                        &= \frac{1}{2}(d_{k}^{T}S^{T}Sd_{k}\alpha^{2}-2d_{k}^{T}Sd_{k}\alpha+d^{T}_{k}d_{k}).
\end{align*}
Since $d_{k}^{T}S^{T}Sd_{k}\geq 0$, if $d_{k}^{T}S^{T}Sd_{k}=0$, the iteration terminates;
otherwise, we could get the exact line search step size
$$\alpha_k=\frac{d^{T}_{k}Sd_{k}}{d^{T}_{k}S^{T}Sd_{k}}.$$

Our new algorithm is presented as follows.
\begin{algorithm}
\caption{A new Uzawa-exact type Algorithm}
\label{alg:Uzawa1}
\begin{algorithmic}
\STATE{Given $y_{0}\in \MR^{m},\varepsilon > 0$ and $k:=0$, solve $Ax_0=f-B^Ty_0$}
\WHILE{Given stopping criteria are not satisfied}
\STATE{Compute $d_{k}=Bx_{k}-Cy_{k}-h$}
\STATE{Solve $Aq_k=B^Td_k$}
\STATE{Set $p_k=Bq_k+Cd_k$ and $\alpha_{k}=\frac{d_{k}^{T}p_k}{p_k^Tp_k}$}
\STATE{Update $y_{k+1}=y_{k}+\alpha_{k}d_{k}$ and $x_{k+1}=x_{k}-\alpha_kq_k$}
\STATE{Set $k:=k+1$}
\ENDWHILE
\end{algorithmic}
\end{algorithm}

From the update of $x_k$ in Algorithm \ref{alg:Uzawa1}, we can get by induction that for any given $k\geq 0,$
\begin{equation*}
\begin{aligned}
x_{k}=A^{-1}(f-B^Ty_{k}),
\end{aligned}
\end{equation*}
which is the same as the equation (\ref{eq:symmetric1a}).

Since the matrix $A$ is sparse and invertible, in this algorithm, the equation $Aq_k=B^Td_k$ could be
solved by using exact LU factorization. All the other computation only includes multiplications with the matrix $B$ and $C$, which can be easily handled. The stopping criteria of the algorithm will be considered in Section \ref{sec:exp}.

\section{Convergence Analysis of the New Algorithm}
\label{sec:convergence}
In this section, we will analyse the convergence of Algorithm \ref{alg:Uzawa1}. The following lemma \cite{cao2004fast} will be used in the analysis.

\begin{lemma}\label{lemma1}
Suppose that A is positive definite and the stabilizing condition
\begin{equation*}
\exists \beta>0,\quad (S_sv,v)\geq \beta\|v\|^2, \quad \forall v\in \MR^m
\end{equation*}
holds, where $S_s=BA^{-1}_sB^T+C$, then
\begin{equation}\label{eq:caoestimation}
(Sv,v)\geq \gamma^{-2}\beta\|v\|^2, \quad \forall v\in \MR^m
\end{equation}
for all $\gamma,$ where $\gamma\geq \|A\|/\lambda_m$ and $\lambda_m=\min\limits_{\lambda\in\sigma(A_s)}\lambda$.
\end{lemma}

\begin{theorem}\label{Thm:convergence}
Suppose that A is positive definite and the stabilizing condition
\begin{equation*}
\exists \beta>0,\quad (S_sv,v)\geq \beta\|v\|^2, \quad \forall v\in \MR^m
\end{equation*}
holds. Let $M=S^TS\succ 0$. For any $y_0$, let the sequence $\{y_k\}$ be generated by Algorithm \ref{alg:Uzawa1} and $y^*$ be the solution of the problem (\ref{eq:quadratic2}),
then there is
\begin{equation}\label{eq:thm1}
\frac{Q(y_{k+1})}{Q(y_{k})}=\frac{\|y_{k+1}-y^*\|_M^2}{\|y_{k}-y^*\|_M^2}\leq 1-c_0<1,
\end{equation}
where $c_0$ is a constant, which is independent of k and the condition number of KKT matrix.
\end{theorem}

\begin{proof}
Since $y^{*}$ is a solution of the problem (\ref{eq:quadratic2}) and $S$ is a invertible square matrix, $y^{*}$ satisfies
$$Sy^{*}=b.$$
Then we obtain
\begin{equation*}
\frac{1}{2}\|y_{k}-y^*\|_M^2=\frac{1}{2}(Sy_k-Sy^*)^T(Sy_k-Sy^*)=Q(y_k)-Q(y^*)=Q(y_k).
\end{equation*}
Considering $d_k=-(Sy_k-b)$ and the exact line search step size $$\alpha_k=\frac{d_k^TSd_k}{d_k^TS^TSd_k},$$
and using the similar arguments from Theorem 5.3 of Y Saad's book \cite{saad2003iterative}, we get
\begin{equation}\label{eq:prove1}
\begin{aligned}
\frac{\|y_{k+1}-y^*\|_M^2}{\|y_{k}-y^*\|_M^2}=1-\frac{(d_k^TSd_k)^2}{(d_k^TS^TSd_k)(d_k^Td_k)}\in [0,1].
\end{aligned}
\end{equation}
To estimate the lower bound of the right side of the above equation, according to Lemma \ref{lemma1} and Cauchy-Schwarz inequality, there is
\begin{equation*}
\begin{aligned}
1\geq\frac{(d_k^TSd_k)^2}{(d_k^TS^TSd_k)(d_k^Td_k)}&=\frac{(Sd_k,d_k)^2}{\|Sd_k\|^2\|d_k\|^2}\geq \frac{\gamma^{-4}\beta^2\|d_k\|^2}{\|Sd_k\|^2}\\
&\geq \frac{\beta^2}{\gamma^4\|S\|^2},\\
\end{aligned}
\end{equation*}
where $\gamma$ is a parameter with
$$\gamma\geq \|A\|/\lambda_m,\mbox{ÆäÖÐ}\lambda_m=\min\limits_{\lambda\in\sigma(A_s)}\lambda,$$
and $\beta$ satisfies
$$\beta\leq \min_{\lambda\in\sigma(S_s)}\lambda.$$
Let $c_0=\frac{\beta^2}{\gamma^4\|S\|^2}\in (0,1],$ then from (\ref{eq:prove1}) and Lemma \ref{lemma1} we get (\ref{eq:thm1}).
Therefore, $c_0$ is independent of k and the condition number of KKT matrix, which concludes the proof.
\end{proof}

In this theorem, the convergence of the sequence $\{y_{k}\}$ is concerned. If we consider the iterative errors for both $\{y_{k}\}$ and $\{x_{k}\}$, we can get the following result.

\begin{theorem}
Suppose that A is positive definite and the stabilizing condition
\begin{equation*}
\exists \beta>0,\quad (S_sv,v)\geq \beta\|v\|^2, \quad \forall v\in \MR^m
\end{equation*}
holds. For any $y_0$, let the sequence $\{(x_k,y_k)\}$ be generated by Algorithm \ref{alg:Uzawa1}, $y^*$ be the solution of the problem (\ref{eq:quadratic2}), and $x^*$ be the corresponding solution of the problem (\ref{eq:symmetric}).
Define the errors $e_{k}^x=x_{k}-x^{*}$, $e_{k}^y=y_{k}-y^{*}$, then
\begin{equation}
\|Be_{k}^x-Ce_{k}^y\| \leq\|Be_0^x-Ce_0^y\|(\sqrt{1-c_0})^k,
\end{equation}
where $c_0$ is a constant, which is independent of k and the condition number of KKT matrix.
\end{theorem}

\begin{proof}
Since $x_{k}=(A^{T}A)^{-1}(A^{T}f-A^{T}B^{T}y_{k})=A^{-1}(f-B^{T}y_{k})$ and $y_{k+1}=y_{k}+\alpha_k(Bx_{k}-Cy_{k}-h)$,
we have
\begin{align*}
 x_{k+1}-x_{k} &=A^{-1}B^{T}(y_{k}-y_{k+1}) \\
         &=-\alpha_k A^{-1}B^{T}(Bx_{k}-Cy_{k}-h).
\end{align*}
Notice that $x_{k+1}-x_{k}=(x_{k+1}-x^{*})-(x_{k}-x^{*})=e_{k+1}^x-e_{k}^x$ and $(x^{*},y^{*})$ is a solution of the problem (\ref{eq:symmetric}), so
\begin{align*}
e_{k+1}^x-e_{k}^x &=-\alpha_k A^{-1}B^{T}(Bx_{k}-Cy_{k}-h) - \alpha_k A^{-1}B^{T}(Bx^{*}-Cy^{*}-h)\\
         &=-\alpha_k A^{-1}B^{T}(Be_{k}^x-Ce_{k}^y),
\end{align*}
such that
\begin{equation}\label{eq:e1}
B(e_{k+1}^x-e_{k}^x)=-\alpha_k BA^{-1}B^{T}(Be_{k}^x-Ce_{k}^y).
\end{equation}
Similarly, we can get
\begin{equation}\label{eq:e2}
y_{k+1}-y_{k}=\alpha_k (Be_{k}^x-Ce_{k}^y),
\end{equation}
and
\begin{equation}\label{eq:e3}
C(e_{k+1}^y-e_{k}^y)=\alpha_k C(Be_{k}^x-Ce_{k}^y).
\end{equation}
From (\ref{eq:e1}) and (\ref{eq:e3}), we obtain
 \begin{align*}
 B(e_{k+1}^x-e_k^x) - C(e_{k+1}^y-e_{k}^y)=-\alpha_k (BA^{-1}B^{T} + C)(Be_{k}^x-Ce_{k}^y),
\end{align*}
so
\begin{equation*}
Be_{k+1}^x - Ce_{k+1}^y  = (I -\alpha_k S)(Be_{k}^x-Ce_{k}^y).
\end{equation*}
Then we have
\begin{equation}\label{eq:e4}
\|Be_{k+1}^x-Ce_{k+1}^y\|^2=[(I-\alpha_kS)(Be_k^x-Ce_k^y)]^T[(I-\alpha_kS)(Be_k^x-Ce_k^y)].
\end{equation}
From the equation (\ref{eq:e2}), we can get
\begin{equation*}
Be_k^x-Ce_k^y=\frac{y_{k+1}-y_k}{\alpha_k}=d_k,
\end{equation*}
therefore
\begin{equation}\label{eq:e5}
\|Be_k^x-Ce_k^y\|^2=d_k^Td_k.
\end{equation}
Combining (\ref{eq:e4}) and (\ref{eq:e5}), we obtain
\begin{equation*}
\begin{aligned}
\frac{\|Be_{k+1}^x-Ce_{k+1}^y\|^2}{\|Be_k^x-Ce_k^y\|^2}&=\frac{d_k^Td_k-2\alpha_kd_k^TSd_k+\alpha_k^2d_k^TS^TSd_k}{d_k^Td_k}.
\end{aligned}
\end{equation*}
Substituting the the exact line search step size $$\alpha_k=\frac{d_k^TSd_k}{d_k^TS^TSd_k}$$
into the above equation, we have
\begin{equation*}
\begin{aligned}
\frac{\|Be_{k+1}^x-Ce_{k+1}^y\|^2}{\|Be_k^x-Ce_k^y\|^2}=1-\frac{(d_k^TSd_k)^2}{(d_k^TS^TSd_k)(d_k^Td_k)}\leq 1-c_0\in[0,1),
\end{aligned}
\end{equation*}
where $c_0$ is the same as that in Theorem \ref{Thm:convergence}. Therefore
\begin{equation*}
\|Be_k^x-Ce_k^y\|\leq\|Be_0^x-Ce_0^y\|(\sqrt{1-c_0})^k,
\end{equation*}
the proof is completed.
\end{proof}

\section{Numerical Experiments}
\label{sec:exp}
In this section, we consider the numerical behavior of Algorithm \ref{alg:Uzawa1} on solving nonsymmetric saddle point problems. We choose two different testing problems. One is linear variational inequalities, which are closely related with optimization problems. The other is finite element approximations of Navier-Stokes equations, which are popular applications of nonsymmetric saddle point problems. In the following two experiments, all computations are performed in MATLAB on a Intel Core i5 PC computer.
\subsection{Linear Variational Inequality}
\label{sec:exp1}
Let $D$ be a closed convex subset of $\MR^n$ and $F(x)$ be a continuous mapping from $\MR^n$ to $\MR^n$. The variational inequality (VI) problems can be stated as
\begin{equation}\label{eq:VI}
    \mbox{Find}~~x^*\in D\quad\mbox{such that} \quad \langle F(x^*),x-x^* \rangle \geq 0 \quad \forall x\in D,
\end{equation}
where $\langle\cdot,\cdot\rangle$ denotes the inner product in $\MR^n$. When $F$ is affine, i.e. $F(x)=Ax-f$ with $A\in\MR^{n\times n}$ and $f\in\MR^n$, we say the VI problem (\ref{eq:VI}) is a \emph{linear variational inequality} problem.

Linear variational inequality problem has been widely used to formulate and study various equilibrium models in the fields of economics, transportation and regional sciences \cite{florian1989mathematical,dafermos1980traffic,nagurney1989general}.

When the coefficient matrix $A$ is positive definite, Mancino and Stampacchia \cite{mancino1972convex} proved that solving the problem (\ref{eq:VI}) is equivalent to solving the quadratic programming problem with constraints
\begin{equation}\label{eq:qp}
\min\limits_{x}~\frac{1}{2}x^{T}Ax-f^{T}x\quad \mbox{s.t.}\quad x\in D.
\end{equation}
If $D$ is polyhedral, i.e. $D=\{x\in\MR^n|Bx=h\}$ with $B\in \MR^{m\times n}$, the KKT point $(x,\lambda)$ of the problem (\ref{eq:qp}) satisfies
    \begin{equation*}
    \begin{bmatrix}
      A & B^{T} \\
      B & 0 \\
    \end{bmatrix}
\left[
         \begin{array}{c}
           x \\
           -\lambda \\
         \end{array}
\right]
=\left[
       \begin{array}{c}
            f \\
            h \\
       \end{array}
 \right],
    \end{equation*}
where $\lambda$ is the Lagrangian multiplier. This is the problem (\ref{eq:symmetric}) with $C=0$.

As an application of Algorithm \ref{alg:Uzawa1} on variational inequalities, we generate nonsymmetric positive definite matrix $A$ with different dimensions. More precisely, we fix $m=n/2$ and generate the matrix $B$, the vectors $f,h$, the initial points by the random function \emph{rand} and \emph{randn} in MATLAB. In the algorithm, the stopping criteria are used as $\frac{\|r_{k}\|}{\|r_0\|}<10^{-6}$ and $|\frac{\|r_{k}\|}{\|r_0\|}-\frac{\|r_{k}\|}{\|r_0\|}|<10^{-7}$, where
$
r_k=\left(
    \begin{array}{cc}
      Ax_k+B^Ty_k-f\\
      Bx_k-Cy_k-h\\
    \end{array}
  \right).
$
The maximum iteration number as a safeguard against an infinite loop is 2000.

Table \ref{table:resultforVI} presents the results of Algorithm \ref{alg:Uzawa1} on linear variational inequality problems. The left half of Table \ref{table:resultforVI} shows the sizes of the problems, the condition numbers of the matrix A (cond(A)) and the condition numbers of KKT matrix (cond(KKT)). The right half of Table \ref{table:resultforVI} shows the $\infty$-norm of the residuals of the solutions, the iteration numbers (Ite) and the CPU time (in seconds) required to solve the problems. From Table \ref{table:resultforVI}, we can see that all of the linear variational inequality problems could be solved within 400 iterations.

\subsection{Navier-Stokes equation}
\label{sec:exp2}
For a more practical application, we consider the numerical behavior of Algorithm \ref{alg:Uzawa1} on solving nonsymmetric saddle point problems arising from finite element
approximations of the steady-state Navier-Stokes equations. The model of the problem is shown as follows,
\begin{equation}\label{eq:Navier}
    \begin{cases}
        -\nu\Delta u + u\cdot \nabla u +\nabla p = f \;in \;\Omega,  \\
        \nabla\cdot u = 0 \; in \; \Omega,
    \end{cases}
\end{equation}
where $\Omega\in \MR^2$ is a bounded domain, $u$ is a vector valued function representing the fluid velocity, and $\nu$ is the
kinematic viscosity of the flow. The scalar function p is the fluid pressure.

First we consider how to generate our problems. The IFISS software library \cite{silvester2012ifiss} is ¡°open-source¡± and is written in MATLAB. The software package of it
is  for the interactive numerical study of incompressible flow problems. It has two important components,
one of which concerns problem specification and finite element discretization. We use this component of IFISS to get nonsymmetric saddle point problems
by finite elements discretization of the steady-state Navier-Stokes equations.

IFISS contains a number of built-in model problems, from which we choose the following
four examples of Navier-Stokes problems with $Q1$-$P0$ element \cite{elman2014finite} in our experiments
\begin{itemize}
  \item Exa 1: Channel domain,
  \item Exa 2: Flow over a backward facing step,
  \item Exa 3: Lid driven cavity,
  \item Exa 4: Flow in a symmetric step channel.
\end{itemize}

The driver \emph{navier\_testproblem} is used to generate our problems. For each example, we choose four different sets of parameters, which means we get four problems from one example. The parameters are given in Table \ref{table:para}, where stab stands for stabilization.

From these 16 problems, we choose problem 4-1 to show its finite element subdivision in Fig.~\ref{pic:fe}, non-zero elements distribution of matrix $A$ and $A-A^T$ in Fig.~\ref{pic:nz}, where nz represents the number of non-zero elements.

Algorithm \ref{alg:Uzawa1} is used to solve these 16 problems. The computing platform and the stopping criteria are the same as the numerical experiment in Section \ref{sec:exp1}, and the starting points are generated by the random function
\emph{rand} in MATLAB.

Table \ref{table:scale} shows the sizes, condition numbers of our problems and the performances of Algorithm \ref{alg:Uzawa1}, where we can find that Exa 3 is very ill-conditioned. From the table, we can see that all of the problems could be solved in 1200 iterations with required accuracy.

Fig.~\ref{pic:exa-1}, \ref{pic:exa-2}, \ref{pic:exa-3} and \ref{pic:exa-4} indicate the performance of the new algorithm in the precision of iterations, where the ordinate axis indicating the logarithm of $\|r_k\|/\|r_0\|$. From those figures and Table \ref{table:scale}, we can see that
\begin{itemize}
\item The algorithm is monotonic, and can converge to a high precision solution for all types of Navier-Stokes equations, i.e. around $10^{-7}$ in the sense of residuals and $10^{-6}$ in the sense of residual ratios.
\item The algorithm is low time consuming, and is not sensitive to the size of the problems, at least in the range we used: $(n,m)=(5890,2816)\sim(n,m)=(8450,4096)$.
\item The algorithm is not sensitive to the ill-conditioning of the problems. Since the condition number of Exa 3 is large, from our results, we can not see its influence on CPU time and precision of the solutions.
\end{itemize}

\section{Conclusions}
\label{sec:con}
In this paper, we propose a new Uzawa-exact type algorithm for the nonsymmetric saddle point problems. The algorithm transforms the original system to a least squares problem. A special descent direction with the exact line search stepsize is chosen to solve the problem. The numerical experiments show that the proposed algorithm is simple and efficient for solving the large-scale nonsymmetric saddle point problems, which arise from linear variational equality problems and Navier-Stokes equations by mixed finite element discretization.

\newpage
\begin{table}[htpb]
  \tbl{The scales of linear VI problems and the performances of Algorithm \ref{alg:Uzawa1}}
    {\begin{tabular}{cccc||ccc} \toprule
           $n$ & $m$& cond(A) & cond(KKT)& $\|r\|_{\infty}$ & Ite  & CPU\\ \midrule
       $1000$ & $500 $& $3357.0 $  &   $1.71\times10^{6}$ &$8.16\times10^{-6}$&$298 $ & $1.314 $ \\ \midrule
       $3000$ & $1500$& $9882.7 $  &   $1.46\times10^{7}$ &$8.42\times10^{-6}$&$334 $ & $11.606 $\\ \midrule
       $5000$ & $2500$& $17214.2$  &   $3.77\times10^{7}$ &$7.12\times10^{-6}$&$246 $ & $24.073 $ \\\midrule
       $7000$ & $3500$& $25343.5$  &   $6.98\times10^{7}$ &$7.71\times10^{-6}$&$348 $ & $71.876 $\\ \midrule
       $10000$ & $5000$& $35454.6$  &   $1.47\times10^{8}$ &$7.46\times10^{-6}$&$335 $ & $228.819$ \\ \bottomrule
    \end{tabular}}
  \label{table:resultforVI}
\end{table}
\begin{table}[hbt]
  \tbl{Parameters chosen in four examples}
    {\begin{tabular}{cccc||cccc}
      \toprule
        & grid  &$\nu$  &stab & & grid  &$\nu$  &stab
         \\ \midrule
       Prob $1-1$& $64\times64$ & $0.01$& $0.25$ & Prob $2-1$& $32\times96$ & $0.02$& $0.25$\\
       Prob $1-2$& $64\times64$ & $0.1$& $1$     & Prob $2-2$& $32\times96$ & $0.2$&  $1$\\
       Prob $1-3$& $64\times64$ & $0.1$& $0.25$  & Prob $2-3$& $32\times96$ & $0.1$& $0.25$\\
       Prob $1-4$& $64\times64$ & $0.05$& $0.25$ & Prob $2-4$& $32\times96$ & $0.05$& $0.25$\\ \midrule
       Prob $3-1$& $64\times64$ & $0.01$& $0.25$ & Prob $4-1$& $32\times96$ & $0.02$ &$0.25$\\
       Prob $3-2$& $64\times64$ & $0.1$& $0.25$  & Prob $4-2$& $32\times96$ & $0.2$  &$1$\\
       Prob $3-3$& $64\times64$ & $0.05$  &$0.5$ & Prob $4-3$& $32\times96$ & $0.5$  &$0.25$\\
       Prob $3-4$& $64\times64$ & $0.02$& $0.25$ & Prob $4-4$& $32\times96$ & $0.01$ &$0.25$\\ \bottomrule
  \end{tabular}}
  \label{table:para}
\end{table}

\begin{table}[htpb]
  \centering
  \tbl{The scales of Navier-Stokes equations and the performances of Algorithm \ref{alg:Uzawa1}}
    {\begin{tabular}{cccccccc}
      \toprule
          & $n$ & $m$& cond(A) & cond(KKT)& $\|r\|_{\infty}$ & Ite  & CPU\\ \midrule
       Prob $1-1$&$8450$ & $4096$& $3208.2$  &   $6.85\times10^{3}$ &$9.59\times10^{-8}$&$582 $ & $2.464$ \\
       Prob $1-2$&$8450$ & $4096$& $1781.4$  &   $4.71\times10^{3}$ &$2.91\times10^{-8}$&$346 $ & $1.548$ \\
       Prob $1-3$&$8450$ & $4096$& $1781.4$  &   $4.78\times10^{3}$ &$1.39\times10^{-8}$&$148 $ & $0.747$ \\
       Prob $1-4$&$8450$ & $4096$& $2061.7$  &   $4.09\times10^{3}$ &$2.59\times10^{-8}$&$214 $ & $0.987$\\ \midrule
       Prob $2-1$&$5890$ & $2816$& $2749.4$  &   $2.98\times10^{3}$ &$2.02\times10^{-7}$&$805 $ & $1.483$ \\
       Prob $2-2$&$5890$ & $2816$& $1453.1$  &   $5.95\times10^{3}$ &$3.58\times10^{-8}$&$933 $ & $1.752$\\
       Prob $2-3$&$5890$ & $2816$& $5089.5$  &   $3.49\times10^{3}$ &$3.68\times10^{-8}$&$355 $ & $0.676 $ \\
       Prob $2-4$&$5890$ & $2816$& $25898.9$  &   $7.83\times10^{3}$ &$1.54\times10^{-7}$&$193 $ & $0.381$ \\ \midrule
       Prob $3-1$&$8450$ & $4096$& $3162.7$ &   $2.85\times10^{16}$&$9.04\times10^{-8}$&$881 $ & $3.602 $  \\
       Prob $3-2$&$8450$ & $4096$& $1664.1$  &   $9.96\times10^{18}$&$6.64\times10^{-9}$&$55  $ & $0.382 $ \\
       Prob $3-3$&$8450$ & $4096$& $1982.1$  &   $1.42\times10^{19}$&$2.04\times10^{-8}$&$182 $ & $0.883 $ \\
       Prob $3-4$&$8450$ & $4096$& $2552.8$ &   $1.08\times10^{18}$&$5.10\times10^{-7}$&$425 $ & $1.802 $ \\ \midrule
       Prob $4-1$&$5890$ & $2816$& $2880.8$  &   $2.89\times10^{3}$ &$1.92\times10^{-7}$&$787 $ & $1.467$ \\
       Prob $4-2$&$5890$ & $2816$& $646.4 $  &   $5.42\times10^{3}$ &$1.47\times10^{-8}$&$409 $ & $0.793$ \\
       Prob $4-3$&$5890$ & $2816$& $667.6 $  &   $2.51\times10^{4}$ &$5.97\times10^{-9}$&$458 $ & $0.871$ \\
       Prob $4-4$&$5890$ & $2816$& $3991.1$  &   $3.88\times10^{3}$ &$3.83\times10^{-7}$&$1185$ & $2.170$ \\ \bottomrule
    \end{tabular}}
    \label{table:scale}
\end{table}

\begin{figure}
\centering
\subfloat[Finite element subdivision]{%
\resizebox*{7cm}{!}{\includegraphics{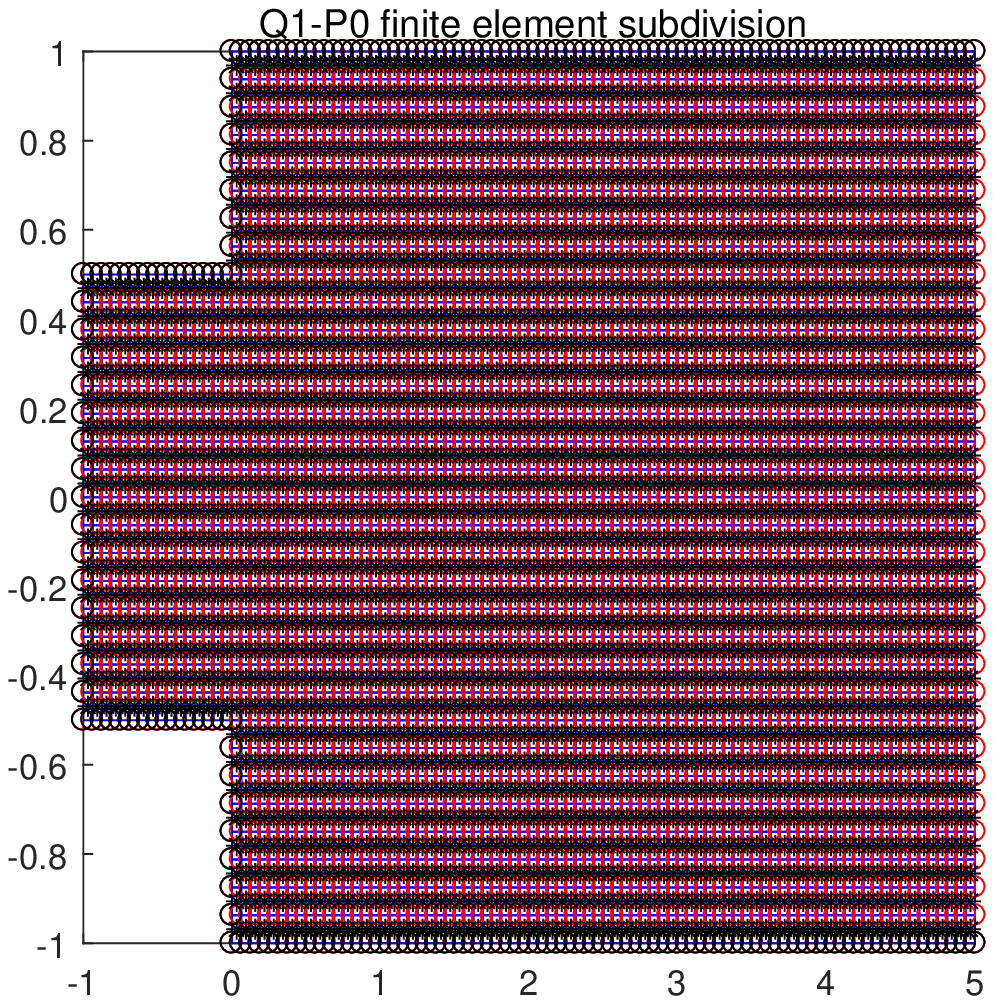}}
\label{pic:fe}
}
\subfloat[Non-zero elements distribution]{%
\resizebox*{7cm}{!}{\includegraphics{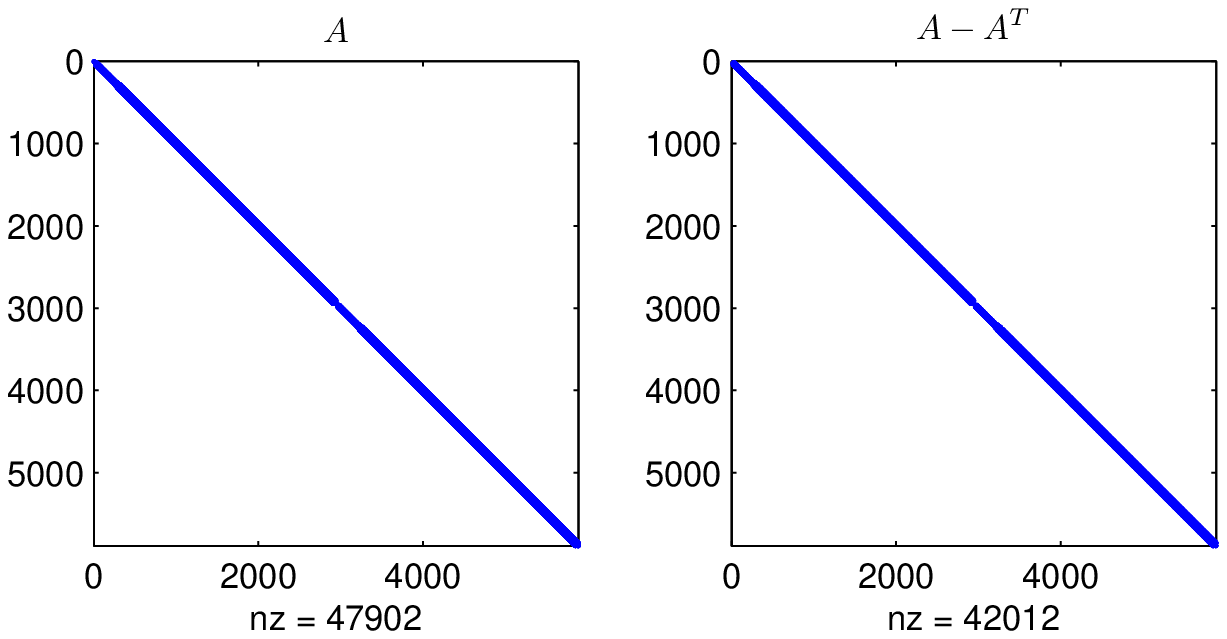}}
\label{pic:nz}
}
\caption{The finite element approximations of Problem 4-1}
\label{pic:pro}
\end{figure}

\begin{figure}
\centering
\subfloat[Performance of the algorithm to Exa 1]{%
\resizebox*{7cm}{!}{\includegraphics{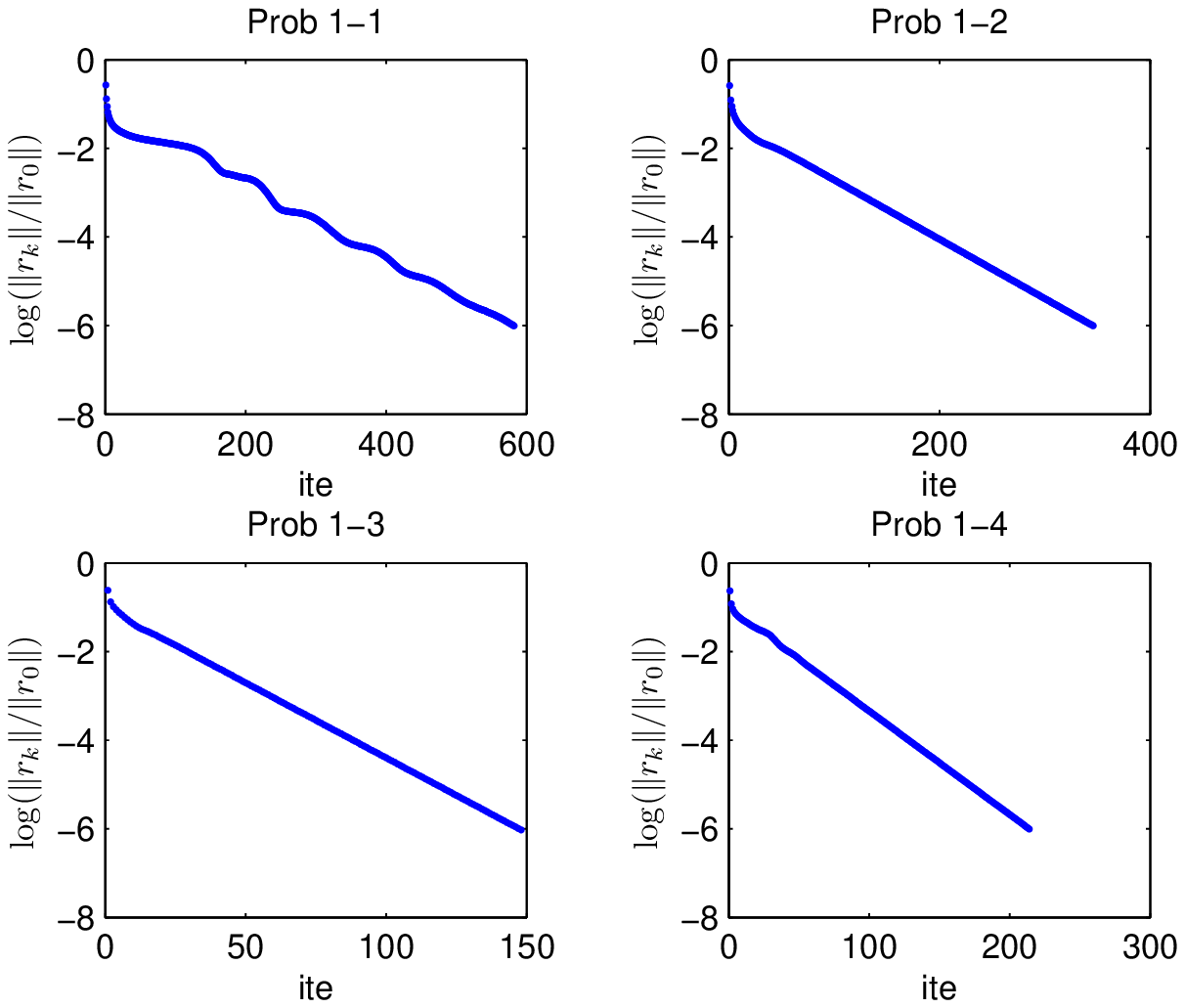}}
\label{pic:exa-1}
}
\subfloat[Performance of the algorithm to Exa 2]{%
\resizebox*{7cm}{!}{\includegraphics{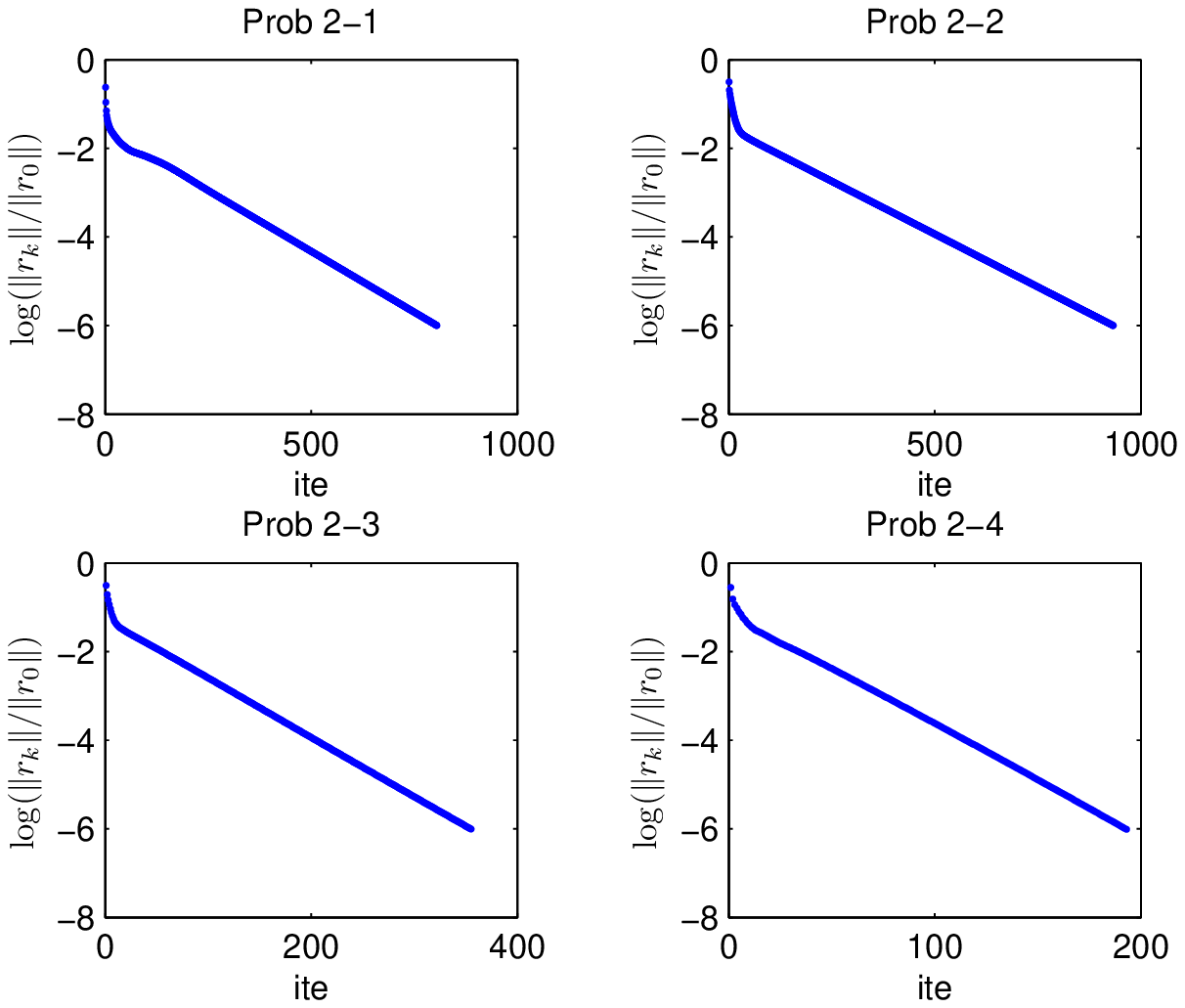}}
\label{pic:exa-2}
}
\caption{The performance of Algorithm \ref{alg:Uzawa1} on Navier-Stokes equations}
\label{pic:test1}
\end{figure}

\begin{figure}
\centering
\subfloat[Performance of the algorithm to Exa 3]{%
\resizebox*{7cm}{!}{\includegraphics{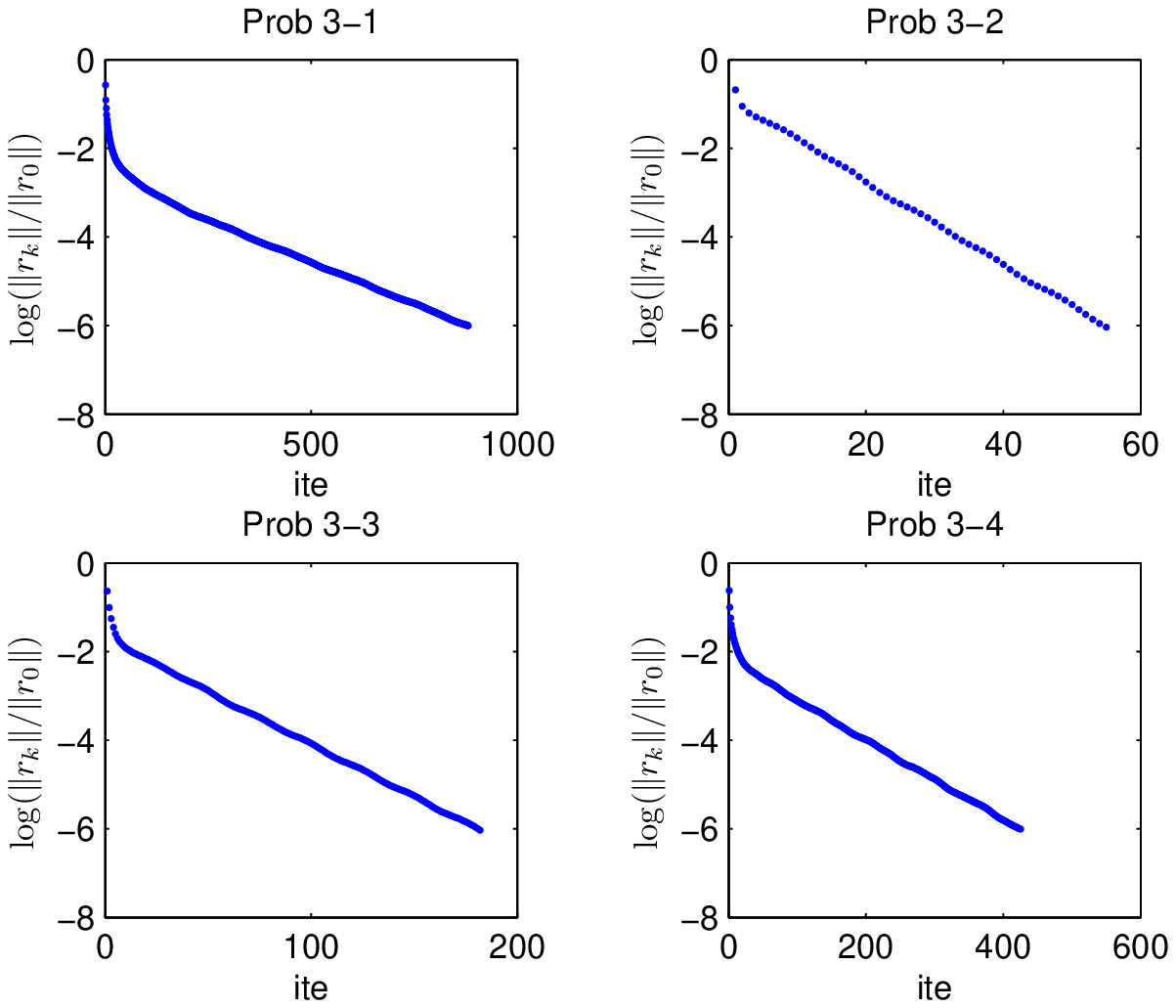}}
\label{pic:exa-3}
}
\subfloat[Performance of the algorithm to Exa 4]{%
\resizebox*{7cm}{!}{\includegraphics{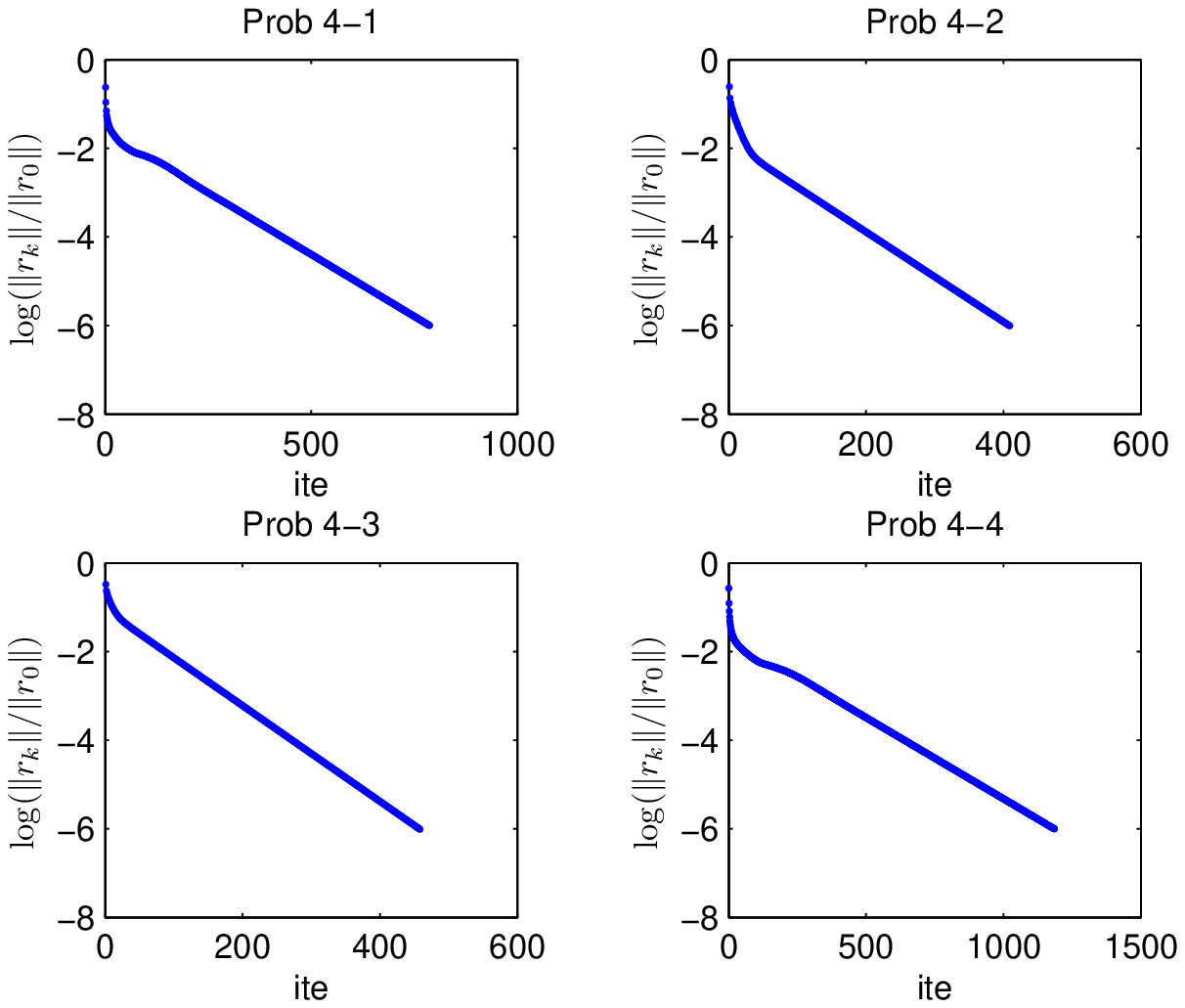}}
\label{pic:exa-4}
}
\caption{The performance of Algorithm \ref{alg:Uzawa1} on Navier-Stokes equations}
\label{pic:test2}
\end{figure}

\end{document}